\documentclass[a4paper,12pt]{article}
\title{{\large \textbf{}} }
\date{}
\usepackage{mathrsfs}
\usepackage{amsfonts}
\usepackage{longtable}
\usepackage{graphicx}
\usepackage{subfigure}
\usepackage{amsthm}
\usepackage{amssymb}
\usepackage{amsmath}
\usepackage[all]{xy}
\usepackage{color}
\usepackage{array}
\usepackage{booktabs}

\makeatletter
\@addtoreset{equation}{section}
\makeatother

\newtheorem{Remark}{\bf{Remark}}
\newtheorem{Example}{\bf{Example}}

\newtheorem{Theorem}{\bf{Theorem}}
\newtheorem{Lemma}{\bf{Lemma}}

\newtheorem{Definition}{\bf{Definition}}

\title{FE-Holomorphic Operator Function Method for Nonlinear Plate Vibrations with Elastically Added Masses}
\author{Xiangying Pang \thanks{Beijing Computational Science Research Center, Beijing 100193, China. ({\tt xypang@csrc.ac.cn}).} \and
Jiguang Sun \thanks{Department of
Mathematical Sciences, Michigan Technological University, Houghton, MI 49931, USA. ({\tt  jiguangs@mtu.edu}).}
\and Zhimin Zhang \thanks{Beijing Computational Science Research Center, Beijing 100193, China and Department of Mathematics, Wayne State University, Detroit, MI 48202, USA. ({\tt zmzhang@csrc.ac.cn}). The research of Z. Zhang is supported partially by the National Natural Science Foundation of China grants NSFC 11871092, NSAF U1930402 and NSFC 11926356.}
}

\begin{document}
\maketitle

\begin{abstract}
Vibrations of structures subjected to concentrated point loads have many applications in mechanical engineering. Experiments are expensive and numerical methods are often used for simulations. In this paper, we consider the plate vibration with nonlinear dependence on the eigen-parameter. The problem is formulated as the eigenvalue problem of a holomorphic Fredholm operator function. The Bogner-Fox-Schmit element is used for the discretization and the spectral indicator method is employed to compute the eigenvalues. The convergence is proved using the abstract approximation theory of Karma  \cite{26, 27}. Numerical examples are presented for validations.	
\end{abstract}


\section{Introduction}

We consider the numerical computation of the natural frequencies of a mechanical structure joined elastically with discrete masses \cite{17}. In structural engineering, the dynamic analysis of structure-spring-load systems, which describe the vibrations of structures such as shells and plates, has been an important research topic in the past a few decades \cite{19,21,8}. Experimental studies are usually expensive and sometimes prohibitive. Instead, numerical methods are often used to compute the vibration characteristics. 

In this paper, we consider the plate-spring-load systems. The contact surface between the oscillators and the plate is small enough and modeled by the $Dirac'$ $\delta$-distribution. 
The dependence on the eigenprarameter is nonlinear. As a consequence, numerical discretizations such as finite element methods or finite difference methods lead to nonlinear matrix eigenvalue problems (NLEVP). 
Many schemes have been proposed for the NLEVP. Iterative methods \cite{33, 5,1,3} such as the Arnoldi method and Jacobi-Davidson method project the original problem onto some subspaces in which eigenvalues are computed as approximations. Linearization techniques transform the problem into an equivalent linear eigenvalue problem \cite{7,28}. These methods may increase the size of the problem, make the eigenvalues more sensitive to perturbations, or even destroy the symmetry structure. For iterative methods, it is difficult to obtain a good initial guess. 

Compared to the linear cases, convergence analysis of numerical methods for the nonlinear eigenvalue problems associated with plate vibrations is much less developed. 
To treat the nonlinearity, most literatures construct a continuous function and a fitting function 
and the eigenvalues are the intersections of these two functions. 
In this paper, a new computational approach is proposed and its convergence is analyzed. 
%Here wee are not concerned about all eigenvalues but instead the eigenvalues in a bounded region of the complex plane. 
We first convert the problem into the eigenvalue problem of a holomorphic Fredholm operator function. Then the Bogner-Fox-Schmit (BFS) finite element is used to discretize the operators. A spectral indicator method is developed to practically compute the eigenvalues. The convergence of the discrete eigenvalues is proved using the abstract approximation theory of Karma \cite{26, 27}.

The rest of the paper is organized as follows. In Section 2, we present the model problem and the abstract approximation theory for the eigenvalue problems of holomorphic Fredholm operator functions. In Section 3, the BFS element is employed to discretize the operators and the convergence of the eigenvalue problem is proved. 
The end of Section 3 is a description of the spectral indicator method to compute the eigenvalues. 
Numerical experiments are presented in Section 4. Finally, we draw some conclusions in Section 5. 
% In our paper, we shall use a generic constant $C$ everywhere different.

\section{Model Problem and Preliminaries}
We consider the plate-spring-load model for the flexural vibration of a thin plate elastically combined with added masses. Assume that the masses are joined by elastic springs with stiffness coefficient $K_j$ at point $x_j \in D,$ where $D \subset \mathbb R^2$ is a Lipschitz polygon. The $j$th pair of mass and spring forms a harmonic oscillator with the vibration frequency $ v_j=\sqrt{K_j /M_j}$, where $M_j$ is the mass. The $\delta$-distribution represents the forcing term. Denote by $\omega(x,t)$ the vertical deflection of the plate at $x\in D$ at time $t$ and $\zeta_j(t) $ the vertical displacement of the $j$th mass $M_j$ at time $t$. The motion of the plate and masses satisfies the following equation \cite{17}:
\begin{subequations}{\label{6}}
	\begin{align}
      Lw(x,t)+\rho d w_{tt}(x,t)+\sum_{j=1}^{p} M_j(\zeta_j)_{tt}\delta(x-x_j) &= 0\quad \text{ in } D, \, x_j \in D ,\\[1mm]
      Bw(x,t)&=0 \quad \text{ on } \partial D, \, \\[1mm]
       M_j(\zeta_j)_{tt}+K_j(\zeta_j(t)-w(x_j,t))&=0\quad \text{for } t>0, \, j=1,...,p, 
        \end{align}
%        
%	\begin{equation}
%		Lw(x,t)+\rho d w_{tt}(x,t)-\sum_{j=1}^{p} M_j(\eta_j)_{tt}\delta(x-x_j) &= 0,\, x_j \in D \text{ in } D,
%	\end{equation}
%	\begin{equation}
%		Bw(x,t)&=0\, \text{ on } \partial D,
%	\end{equation}
%	\begin{equation}
%		M_j(\eta_j)_{tt}+K_j(\eta_j(t)-w(x_j,t))&=0,\,t>0, \,j=1,...,p, 
%	\end{equation}
\end{subequations}
where $L$ is the plate operator and $B$ is some boundary operator. 
The parameters $\rho(x),\nu(x),$ $E(x),d(x)$ denote the volume mass density, Possion ratio, Young's modulus, thickness of the plate, respectively, 
such that 
\begin{displaymath}
	K_j>0, \, M_j>0 , \, 0< \nu(x) <1/2 , 
\end{displaymath} 
\begin{displaymath}
	|E(x)| \leq C , \, |\rho(x) | \leq C, \, | d(x) | \leq C,
\end{displaymath}
where $C$ is a constant. 
The plate operator $L$ is given by
\begin{equation}
	L(x)=\partial_{xx}R(x)(\partial_{xx}+\nu (x) \partial_{yy}) +\partial_{yy} R(x)(\partial_{yy}+\nu(x) \partial_{xx}) +2 \partial_{xy} R(x)(1-\nu(x))\partial_{xy},
\end{equation}
where $R(x)=Ed^3/{12(1-\nu^2)}$ is the flexural rigidity of the material.

 The eigenvibrations of this system are characterized by the ansatz:
\begin{equation}{\label{2}}
	w(x,t) = u(x) v(t) ,\, \zeta_j(t) = c_j u(x_j) v(t), \, x \in D, \, j= 1, \dots,p, 
\end{equation}
where $v(t)=a_0 cos(\sqrt{\lambda}t)+b_0 sin(\sqrt{\lambda}t )\,,t>0,$ and $ a_0,\, b_0,\, c_j$ are constants.
Using (\ref{2}) in (\ref{6}), we obtain the nonlinear eigenvalue problem to
find $\lambda \in \mathbb{C} \setminus \cup_j \{ \sigma_j \}$ and nontrivial function $u(x)$ such that 
\begin{eqnarray}
\label{Lux}	 &&Lu(x)=\lambda\rho d u + \sum_{j=1}^p \frac{\lambda \sigma_j}{\sigma_j - \lambda} M_j \delta(x-x_j)u \quad \text{ in } D ,\\
\label{Bux} &&Bu(x)=0 \quad  \text{ on } \partial D,
\end{eqnarray}
where $\sigma_j = {v_j}^2=K_j/M_j$.
In the rest of the paper, we use the clamped plate condition, i.e., $u = \frac{\partial u}{\partial n} =0$, where $n$ is the unit outward normal to $\partial D$.
Without loss of generality, we set $p=1$ and define the bilinear forms
\begin{eqnarray}
\nonumber a( u, v)&=&\int_{D} R(x) [ (\partial_{xx} u +\partial_{yy} u  )(\partial_{xx} v +\partial_{yy} v) \\
&& \label{auv} \qquad \qquad + (1-\nu )(2 \partial_{xy} u \partial_{xy} v - \partial_{xx} u \partial_{yy} v -\partial_{yy} u \partial_{xx} v)] dx, \\
\label{buvlambda}b(u, v;\lambda)&=&\lambda \int_{D} \rho d u v dx - \frac{\lambda \sigma}{\lambda-\sigma} M u(x_0)v(x_0), \quad x_0 \in D.
\end{eqnarray}
% Let $ V:=H^2_0(D)$.

The variational formulation for the eigenvalue problem is to find $\lambda \in \mathbb{C} \setminus \{ \sigma \}$ and nontrivial $u \in H^2_0(D)$ such that
\begin{equation}{\label{1}}
a(u,v)= b(u,v;\lambda),\quad \forall v \in H^2_0(D).
\end{equation}

The existence of a countable set of real eigenvalues for \eqref{1} is proved in \cite{17} using an auxiliary parameter eigenvalue problem and a nonlinear algebraic equation. 
In this paper, we shall take a different approach by writing \eqref{1} as the eigenvalue problem of a holomorphic Fredholm operator function, which is then discretized
by a finite element method. The convergence of the eigenvalues is then proved using the abstract approximation theory of Karma \cite{26,27}.

Now we present some preliminaries from  \cite{26,27}.
The materials are adapted for the finite element method we shall use to discretize \eqref{1}.

 \begin{Definition}
 	Let $V,\, W$ be Banach spaces. A bounded linear operator $F \in \mathcal{L}(V,W) $ is called Fredholm with index zero if
	\begin{itemize}
 	\item[1.]  the range of $F$, denoted by $\mathcal{R}(F)$, is closed and $codim \mathcal{R}(F) := dim (Y/{\mathcal{R}(F))} $ is finite;
 	\item[2.]  the null space of $F$, denoted by $\mathcal{N} (F) $, is finite-dimensional; and
 	\item[3.]  the Fredholm index, defined as $ ind(F) = dim \mathcal{N} (F) - codim \mathcal{R} (F)$, is zero.
 	\end{itemize}
\end{Definition}

\begin{Definition}
	(ref. \cite{26}) Let $V $ and $V_n, \, n \in \mathbb{N},$ be Banach spaces. A sequence of linear operators $\mathcal{P} := \{p_n: V \to V_n\}_{n \in \mathbb{N}}$ connecting them is such that $ \Vert p_n u \Vert_{V_n} \to \Vert u \Vert_V$ for all $u \in V$. 
\end{Definition}
A sequence $\{v_n\}_{n \in \mathbb{N}}$, $v_n \in V_n$, is called $\mathcal{P}$-converging to an element $u \in V$, denoted by $v_n \stackrel{\mathcal{P} }{\to} u $, if $\Vert p_n u-v_n \Vert_{V_n} \to 0, n \to \infty$.
 A sequence $\{v_n\}_{n \in \mathbb{N}}$, $v_n \in V_n$, is called $\mathcal{P}$-compact if for every subsequence $\{v_n\}_{n \in \mathbb{N}'}, \, \mathbb{N}' \subseteq \mathbb{N},$ there exists $ \mathbb{N}'' \subseteq \mathbb{N}'$ and $ u \in V$ such that $v_n \stackrel{\mathcal{P} }{\to} u,\, n \in \mathbb{N}'' \to \infty$.

Let $F: \Omega \to \mathcal{L}(V, W)$ be a holomorphic operator function on $\Omega \subset \mathbb C$ and, for each $\eta \in \Omega$, 
$F(\eta)$ is a Fredholm operator of index zero. 
\begin{Definition} A complex number $\lambda \in \Omega$ is called an eigenvalue
of $F$ if there exists a nontrivial $x \in V$ such that $F(\lambda) x = 0$. The element $x$ is called an eigenelement associated with $\lambda$.
\end{Definition}
The resolvent set $\rho(F)$ and the spectrum $\sigma(F) $ of $F$ are respectively defined as
\begin{equation}\label{rhoT}
\rho(F) = \{\eta \in \Omega: F(\eta)^{-1} \text{ exists and is bounded}\}
\end{equation}
and
\begin{equation}\label{sigmaT}
\sigma(F) = \Omega \setminus \rho(F).
\end{equation}

Since $F(\eta)$ is holomorphic, the spectrum $\sigma(F)$ has no cluster points in $\Omega$ and every $\lambda \in \sigma(F)$ is
an eigenvalue for $F$. Furthermore,
$F^{-1}(\cdot)$ is meromorphic (see Section 2.3 of \cite{27}).
The dimension of $\mathcal{N}(F(\lambda))$ is called the 
geometric multiplicity for an eigenvalue $\lambda$. 

\begin{Definition}
An ordered sequence of elements $x_0, x_1, \ldots, x_k$ in $V$ is called a Jordan
chain of $F$ at an eigenvalue $\lambda$ if
\[
F(\lambda)x_j + \frac{1}{1!} F^{(1)}(\lambda)x_{j-1}+ \ldots + \frac{1}{j!} F^{(j)}(\lambda) x_0 = 0, \quad j = 0, 1, \ldots, k,
\]
where $F^{(j)}(\lambda)$ denotes the $j$th derivative of $F(\lambda)$. 
\end{Definition}

The length of any Jordan chain of an eigenvalue is finite.
Denote by $m(F,\lambda, x_0)$ the length of a Jordan chain formed by an eigenelement $x_0$.
The maximal length of all Jordan chains of the eigenvalue $\lambda$ is denoted by $\kappa(F, \lambda)$.
Elements of any Jordan chain of an eigenvalue $\lambda$ are called generalized eigenelements of $\lambda$.
\begin{Definition}
The closed linear hull of all generalized eigenelements of $F(\cdot)$ at an eigenvalue $\lambda$, denoted by $G(F,\lambda)$,
is called the generalized eigenspace of $\lambda$ for $F(\cdot)$.
\end{Definition}
The following abstract convergence theorem was proved in \cite{27}.
	\begin{Theorem}{\label{Karma}}
		Let $\Omega \subset \mathbb C $ be open, bounded, and simply connected. 
		Denoted by $V$ and $\{V_n \}_{n \in \mathbb{N}}$ a Banach space and a sequence of Banach spaces, respectively. 
		They are connected by a sequence of linear operators $\mathcal{P} := \{p_n: V \to V_n\}_{n \in \mathbb{N}}$ 
		satisfying $ \Vert p_n u \Vert_{V_n} \to \Vert u \Vert_V, \forall u \in V $. Let $F(\cdot): \Omega \to \mathcal{L}(V,V)$ and 
		$F_n(\cdot): \Omega \to \mathcal{L}(V_n,V_n),\, n \in \mathbb{N},$ be holomorphic Fredholm operator functions. 
		Assume $\rho(F) \ne \emptyset$.
Let the following properties be satisfied:
		\begin{itemize}
		\item[(A1)] the operator $F_n(\lambda)$ is a Fredholm operator with index zero for any $\lambda \in \Omega$ and $n\in \mathbb{N}.$
		\item[(A2)] the sequence $\{F_n(\cdot)\}_{n \in \mathbb{N}}$ is equibounded on every compact set $\Omega' \subset \Omega, \,$ namely,
		for a compact set $\Omega' \subset \Omega$, there exists a constant C such that $\Vert F_n(\lambda) \Vert_{\mathcal{L}(V_n,V_n)}  \leq C$ for all
		$n\in \mathbb{N}$ and all $\lambda \in \Omega'$.
		\item[(A3)] for each $\lambda \in \Omega$, the sequence $\{F_n(\lambda)\}_{n \in \mathbb{N}}$ approximates $F(\lambda)$, i.e.,
		\begin{equation}
			\Vert (F_n(\lambda)p_n-p_n F(\lambda)) u \Vert_{V_n} \to 0,  \, \forall  u \in V.
		\end{equation}
		\item[(A4)] the sequence $\{F_n(\lambda)\}_{n \in \mathbb{N}}$ is regular for each $\lambda \in \Omega$, i.e., 
		\begin{equation}
			\{F_n(\lambda)x_n\}_{n \in \mathbb{N}} \textrm{ is $ \mathcal{P-}compact$ } \Rightarrow  \{ x_n\}_{n \in \mathbb{N}} \textrm{ is $ \mathcal{P-}compact$ }, \, \forall \lambda \in \Omega.
		\end{equation}
		\end{itemize}
		If $\lambda_0$ be an eigenvalue of $F(\lambda)$ and $\Omega_0 \subset \Omega$ be a simply-connected compact set  with boundary $\partial \Omega_0 \subset \rho(F) $ and $ \Omega_0 \cap \sigma (F) = \{ \lambda_0 \} $, then the following estimation holds 
		\begin{equation}
			\vert \lambda_n -\lambda_0 \vert \leq c \epsilon_n^{1/\kappa} \to 0, \, \forall \lambda_n \in \sigma(F_n) \cap \Omega_0, 
		\end{equation}
		where $ 	\epsilon_n = \sup_{\eta \in \Gamma} \max_{g \in G(F,\lambda_0), \, \Vert g \Vert_V=1} \Vert F_n(\eta)p_n g - p_n F(\eta)g \Vert_{V_n} ,$ and $ \kappa = \kappa(F,\lambda_0)$. 
		
	\end{Theorem}

% ------------------------------------------------------------------------------------------------------------------------------------------------------
\section{FE-Holomorphic Operator Function Method}

We now employ the $C^1$ BFS finite element \cite{12} to discretize (\ref{1}) and analyze the convergence of the eigenvalues.
For simplicity, we assume that $D$ is a Lipschitz polygon whose boundary segments are parallel to x- or y- axis. 
Let $\Omega$ be a connected compact set in $\mathbb{C} \setminus \{ \sigma \}$.
Let $\|\cdot\|$ denote the usual $L_2$-norm on a Sobolev space. 
Note that $a(\cdot,\cdot)$ defines an inner product $(w,u)_V=a(w,u)$ on $V:=H_0^2(D)$. The associated energy norm 
$\Vert u \Vert_V = \sqrt{a(u,u)} $ is equivalent to the usual Sobolev norm defined on $H_0^2(D)$ (see, e.g., \cite{29}). 
%This can be clearly observed by the fact that the semi-norm $\vert \cdot \vert_2$ is actually equivalent to norm $ \Vert \cdot \Vert_2$ in $H_0^2(\Omega)$. 
% Bogner-Fox-Schmit element, which is $C^1$ conforming, is implemented to discrete the variation problem. 
Let $\mathcal{T}_h$ be a regular rectangular mesh where $h$ is the mesh size. 
Denote $V_h$ the corresponding BFS finite element space satisfying the boundary conditions. 
We require that $x_0$ coincides with some mesh node.

The discrete problem for (\ref{1}) is to
find $\lambda_h \in \Omega $ and nontrivial functions $u_h \in V_h $ such that
\begin{equation}{\label{discrete}}
	a(u_h,v_h)= b(u_h,v_h;\lambda_h), \quad \forall v_h \in V_h.
\end{equation}

\begin{Lemma}
	For a fixed $\lambda \in \Omega$, $ b(\cdot,\cdot;\lambda)$ is a bounded bilinear form on $V \times V$.  
\end{Lemma}
\begin{proof}
Clearly, $ b(\cdot,\cdot;\lambda)$ is a bilinear form for $\lambda \in \Omega$. For $u,v \in V$ and $ \lambda \in \Omega$, for $s>0$, we have that
\begin{eqnarray} \nonumber
\vert b(u,v;\lambda) \vert  & = & \left\vert \lambda \int_{D} \rho d u v dx - \frac{\lambda \sigma}{\lambda-\sigma} M u(x_0)v(x_0) \right\vert \\
\nonumber		& \leq & \lambda \rho d \Vert u \Vert \Vert v \Vert + \left\vert \frac{\lambda \sigma M }{\lambda-\sigma} \right\vert \Vert u \Vert_{0} \Vert u \Vert_{0}\,\\
			%  (\vert u \vert_{0,\infty} \lesssim \Vert u \Vert_2 \, \textrm{and poincar$\acute{e}$ inequality} ) 
\nonumber			& \leq & C(\lambda,D, \rho,d,\sigma,M) \Vert u \Vert_{1+s} \Vert v \Vert_{1+s}\\
\label{bound}			& \leq & C(\lambda,D, \rho,d,\sigma,M) \Vert u \Vert_{V} \Vert v \Vert_{V},
\end{eqnarray}
%	\begin{equation}{\label{bound}}
%		\begin{array}{lll}
%			\vert b(u,v;\lambda) \vert  & = & \vert \lambda \int_{D} \rho d u v dx - \frac{\lambda \sigma}{\lambda-\sigma} M u(x_0)v(x_0) \vert \\
%			& \leq & \lambda \rho d \Vert u \Vert \Vert v \Vert + \vert \frac{\lambda \sigma M }{\lambda-\sigma} \vert \Vert u \Vert_{0,\infty} \Vert u \Vert_{0,\infty}\,\\
%			%  (\vert u \vert_{0,\infty} \lesssim \Vert u \Vert_2 \, \textrm{and poincar$\acute{e}$ inequality} ) 
%			& \leq & C(\lambda,D, \rho,d,\sigma,M) \Vert u \Vert_{1+s} \Vert v \Vert_{1+s}\\
%			& \leq & C(\lambda,D, \rho,d,\sigma,M) \Vert u \Vert_{V} \Vert v \Vert_{V},
%		\end{array}
%	\end{equation}
where $\Vert \cdot \Vert_{0}$ is the norm on $C(D)$ and
the embedding $ H^{1+s}(D) \hookrightarrow C(D)$ for any small $s>0$ is used. 
\end{proof}

We define the operator functions $B(\lambda): \mathbb{C} \setminus \{ \sigma \} \rightarrow \mathcal{L}(V, V)$ such that 
\begin{equation}{\label{B}}
  b(u,v;\lambda)=(B(\lambda)u,v)_V , \quad \forall u,v \in V,
\end{equation}
and $B_h(\lambda): \mathbb{C} \rightarrow \mathcal{L}(V_h, V_h) $ such that
\begin{equation}{\label{Bh}}
    b(u_h,v_h;\lambda)=(B_h(\lambda)u_h,v_h)_V ,\quad \forall u_h,v_h \in V_h .
\end{equation}
Since $V_h \subset V$, the following Galerkin orthogonality holds 
\begin{equation}
((B(\lambda) - B_h(\lambda)) u_h,v_h)_V =0 , \quad \forall u_h,v_h \in V_h.
\end{equation}

Now we define the nonlinear operator function
\begin{equation}\label{F}
F(\lambda):=I-B(\lambda)
\end{equation} 
and its finite element approximation
\begin{equation}\label{Fh}
F_h(\lambda):=I_h-B_h(\lambda).
\end{equation}
% The reason for this is that our nonlinear eigenproblems can be transformed into clearer forms. 

The eigenvalue problems (\ref{1}) and (\ref{discrete}) can be written, respectively, as to
find $\lambda \in \mathbb{C} \setminus \{ \sigma \}$ and $u \in V \setminus \{0\}$ such that
\begin{equation}{\label{trans1}}
	F(\lambda)u=0
\end{equation} 
and find $\lambda_h \in \Omega$ and $u_h \in V_h\setminus \{0\}$ such that
\begin{equation}{\label{trans2}}	F_h(\lambda_h)u_h=0.
\end{equation}

\begin{Lemma}
For $\lambda \in \Omega$ and $h >0$ small enough, $B(\lambda) $ and $B_h(\lambda)$ are compact. 
Furthermore, $F(\lambda)$ and $F_h(\lambda)$ are Fredholm operators with index zero.
\end{Lemma}
\begin{proof}
	Note that given a bounded sequence $\{v_n\}_{n\in \mathbb{N}} \in V$, there is a convergent subsequence in $C(D)$ and, for simplicity, we still denote it by $\{v_{n}\}_{n\in \mathbb{N}} $. By the definition of $B(\lambda)$, we have that
\begin{eqnarray}
\nonumber \Vert B(\lambda)v_{n_1}- B(\lambda) v_{n_2} \Vert_V^2 
\nonumber			& = & (B(\lambda) v_{n_1} - B(\lambda) v_{n_2}, B(\lambda) v_{n_1} - B(\lambda) v_{n_2})_V \\
\nonumber			& = & b(v_{n_1}-v_{n_2},B(\lambda) v_{n_1} - B(\lambda) v_{n_2};\lambda) \\
\nonumber			& = &  \lambda \int_{D} \rho d (v_{n_1}-v_{n_2}) (B(\lambda) v_{n_1} - B(\lambda) v_{n_2} )dx  \\
\nonumber			& & \qquad -\frac{\lambda \sigma}{\lambda-\sigma} M (v_{n_1}-v_{n_2})(x_0)(B(\lambda) v_{n_1} - B(\lambda) v_{n_2})(x_0)  \\
\nonumber			& \leq & C(\lambda,D,\rho,d) \Vert v_{n_1}-v_{n_2} \Vert_{0} \Vert B(\lambda)v_{n_1}- B(\lambda) v_{n_2} \Vert_{0}  \\
\label{compact}			& & \qquad+ C(\lambda,\sigma,M) \Vert v_{n_1}-v_{n_2} \Vert_{0} \Vert B(\lambda)v_{n_1}- B(\lambda) v_{n_2} \Vert_{0}.
\end{eqnarray}
%	\begin{equation}{\label{compact}}
%		\begin{array}{lll}
%			\Vert B(\lambda)s_{n_1}- B(\lambda) s_{n_2} \Vert_V^2 
%			& = & (B(\lambda) s_{n_1} - B(\lambda) s_{n_2}, B(\lambda) s_{n_1} - B(\lambda) s_{n_2})_V \\
%			& = & b(s_{n_1}-s_{n_2},B(\lambda) s_{n_1} - B(\lambda) s_{n_2};\lambda) \\
%			& = &  \lambda \int_{D} \rho d (s_{n_1}-s_{n_2}) (B(\lambda) s_{n_1} - B(\lambda) s_{n_2} )dx  \\
%			& & -\frac{\lambda \sigma}{\lambda-\sigma} M (s_{n_1}-s_{n_2})(x_0)(B(\lambda) s_{n_1} - B(\lambda) s_{n_2})(x_0)  \\
%			& \leq & C(\lambda,D,\rho,d) \Vert s_{n_1}-s_{n_2} \Vert_{0,\infty} \Vert B(\lambda)s_{n_1}- B(\lambda) s_{n_2} \Vert_{0,\infty}  \\
%			& &+ C(\lambda,\sigma,M) \Vert s_{n_1}-s_{n_2} \Vert_{0,\infty} \Vert B(\lambda)s_{n_1}- B(\lambda) s_{n_2} \Vert_{0,\infty}. \\
%		\end{array}  
%	\end{equation} 

Since $\Vert v \Vert_{0} \leq C\Vert v \Vert_V$ for $v \in V$, it holds that
		\begin{equation}
				\Vert B(\lambda)v_{n_1}- B(\lambda) v_{n_2} \Vert_V \leq 	
			C \Vert v_{n_1}-v_{n_2} \Vert_{0} 
			 \to  0.
	\end{equation} 
Thus $\{ B(\lambda)v_n\}_{n\in \mathbb{N}}$ converges. 
Consequently, $B(\lambda)$ is compact and $F(\lambda)$ is a Fredholm operator with index zero.
The same argument holds for $B_h(\lambda)$ and $F_h(\lambda)$. 
\end{proof}

\begin{Lemma}{\label{equibounded}}
	There exists a constant $C$ such that
	$ \Vert F_h(\lambda) \Vert_{\mathcal{L}(V_h,V_h)} \leq C$ for $\lambda \in \Omega, \, h>0 .  $ 
\end{Lemma}
\begin{proof}

It can be seen from (\ref{bound}) and (\ref{Bh}) that
\begin{displaymath}
		\Vert B_h(\lambda) \Vert_{\mathcal{L}(V_h, V_h)} = \sup\limits_{u_h,v_h \in V_h} \frac{\vert (B_h(\lambda)u_h,v_h)_{V_h} \vert}{\Vert u_h \Vert_{V_h} \Vert v_h \Vert_{V_h} } = 
	\sup\limits_{u_h,v_h \in V_h} \frac{\vert b(u_h,v_h;\lambda) \vert}{\Vert u_h \Vert_{V_h} \Vert v_h \Vert_{V_h} }  \leq C(\lambda,D, \rho,d,\sigma,M).
\end{displaymath}

Since $\Omega$ is compact, for $\lambda \in \Omega, \, v_h \in V_h$, one has that
\begin{equation}
	\Vert F_h(\lambda) v_h \Vert_{V_h} = \Vert v_h - B_h(\lambda) v_h \Vert_{V_h} \leq \Vert v_h \Vert_{V_h} + \Vert B_h(\lambda) v_h \Vert_{V_h} \leq C \Vert v_h \Vert_{V_h},
\end{equation}
where $C$ is independent of $\lambda$ and $h$. 
\end{proof}

Define the linear projection operator $p_h: V \to V_h $ such that 
\begin{equation}\label{ph}
a(u-p_h u, v_h) = 0, \, \forall v_h \in V_h.
\end{equation}
% Clearly, $p_h$ is linear and satisfies the conditions in Definition 2. 
% Note that the importance of $p_h$ cannot be neglected in the following part.

\begin{Remark}
	For the biharmonic equation
	\begin{subequations}{\label{DE}}
	\begin{align}
       \Delta ^2 u &= f \quad \text{in} \quad D,\\[1mm]
       u = \frac{\partial u}{\partial n} &= 0 \quad \text{on} \quad \partial D,
        \end{align}
	\end{subequations}
	if the largest interior angle of the boundary $\partial D$ is less than $126.28^{\circ}$ and $f \in H^{-1}(D)$, then $ u \in H^3(D) \cap H_0^2(D) $\cite{14,22}. In fact, Dirac's delta function $\delta(x) \in H^{-1-\epsilon}(D)$ for any $\epsilon > 0$. Hence one can assume that $ u \in H^{3 - \epsilon}(D)$, for any $\epsilon > 0$ small enough \cite{BS2005, 15}.
\end{Remark}

\begin{Lemma}{\label{order}}
	Let $\lambda \in \Omega$ and $v \in H_0^2(D) \cap H^{3-\epsilon}(D)$ for any small $\epsilon > 0$. 
	It holds that 
	 \begin{equation}\label{Re2}
	 \Vert p_h F(\lambda)v-F_h (\lambda) p_h v \Vert_{V_h} \leq C h^{2-\epsilon-s} \Vert v \Vert_{3-\epsilon},
	 \end{equation}
	 where $s >0 $ and small enough.
\end{Lemma}
\begin{proof}
Using (\ref{B}), (\ref{Bh}), (\ref{F}) and (\ref{Fh}),
\begin{eqnarray}
\nonumber |(p_h F(\lambda)v-F_h (\lambda) p_h v, v_h)_{V_h}| &  =  & |(p_h(I-B(\lambda))v-(I_h-B_h(\lambda))p_h v,v_h)_{V_h}| \\
\nonumber			& = & |(-p_h B(\lambda) v + B_h(\lambda) p_h v,v_h)_{V_h}| \\ 
\nonumber			& = & |-(B(\lambda)v, v_h)_{V_h} +(B_h(\lambda)p_h v, v_h)_{V_h}| \\
\nonumber			& = & |-b(v,v_h;\lambda) +b(p_h v, v_h;\lambda) |\\
\label{esti}			& = & |b(p_h v - v,v_h;\lambda) |,
\end{eqnarray}
%	\begin{equation}{\label{esti}}
%				\begin{array}{lll}
%			|(p_h F(\lambda)v-F_h (\lambda) p_h v, v_h)_{V_h}| &  =  & |(p_h(I-B(\lambda))v-(I_h-B_h(\lambda))p_h v,v_h)_{V_h}| \\
%			& = & |(-p_h B(\lambda) v + B_h(\lambda) p_h v,v_h)_{V_h}| \\ 
%			& = & |-(B(\lambda)v, v_h)_{V_h} +(B_h(\lambda)p_h v, v_h)_{V_h}| \\
%			& = & |-b(v,v_h;\lambda) +b(p_h v, v_h;\lambda) |\\
%			& = & |b(p_h v - v,v_h;\lambda) |,
%		\end{array}
%	\end{equation}
where $v_h \in V_h$ and $ \Vert v_h \Vert_{V_h}=1$. In view of (\ref{bound}) and the orthogonal projection theorem \cite{31,32}, (\ref{esti}) can be written as 
\begin{eqnarray}
\nonumber |(p_h F(\lambda)v-F_h (\lambda) p_h v, v_h)_{V_h}| 
\nonumber			& \leq &  C  \Vert v_h \Vert_{V_h} \Vert p_h v-v\Vert_{1+s}  \\
\label{1+s}			& \leq &  C  \Vert p_h v-v\Vert_{1+s} \\
 & \leq & C h^{2-\epsilon-s} \Vert v \Vert_{3-\epsilon}.
\end{eqnarray}
%	\begin{equation}
%	\begin{array}{lll}{\label{1+s}}
%		|(p_h F(\lambda)v-F_h (\lambda) p_h v, v_h)_{V_h}| 
%			& \leq &  C  \Vert v_h \Vert_{V_h} \Vert p_h v-v\Vert_{1+s}  \\
%			& \leq &  C  \Vert p_h v-v\Vert_{1+s}.  \\
%		\end{array}
%	\end{equation}

Hence
\begin{equation}
	 \Vert p_h F(\lambda)v-F_h (\lambda) p_h v \Vert_{V_h} =\sup_{v_h \in V_h, \, \Vert v_h \Vert _{V_h}=1} 	|(p_h F(\lambda)v-F_h (\lambda) p_h v, v_h)_{V_h}| \leq C h^{2-\epsilon-s} \Vert v \Vert_{3-\epsilon}
\end{equation}
and the proof is complete.
\end{proof}

\begin{Lemma}{\label{vh}}
	Let $\lambda \in \Omega$ be fixed. For $v_h \in V_h$, 
	\[
	\Vert (F(\lambda)-F_h(\lambda))v_h \Vert_V = \Vert (B(\lambda)-B_h(\lambda))v_h \Vert_V \to 0 \quad \text{as } \quad h \to 0.
	\]
\end{Lemma}
\begin{proof}
    From (\ref{B}) and (\ref{Bh}), we have $ ((B(\lambda)-B_h(\lambda))u_h,v_h)_V =0$ for all $u_h,\, v_h \in V_h$ and
\begin{eqnarray*}
\Vert (B(\lambda)-B_h(\lambda))v_h \Vert_V^2 & = & ((B(\lambda)-B_h(\lambda))v_h,(B(\lambda)-B_h(\lambda))v_h)_V \\
			& = &  ((B(\lambda)-B_h(\lambda))v_h,(B(\lambda)-p_h B(\lambda))v_h)_V\\
			& \leq & 	\Vert (B(\lambda)-B_h(\lambda))v_h \Vert_V \, \Vert (I-p_h)(B(\lambda)v_h) \Vert_V. 
\end{eqnarray*}
%	\begin{equation}
%		\begin{array}{lll}
%			\Vert (B(\lambda)-B_h(\lambda))v_h \Vert_V^2 & = & ((B(\lambda)-B_h(\lambda))v_h,(B(\lambda)-B_h(\lambda))v_h)_V \\
%			& = &  ((B(\lambda)-B_h(\lambda))v_h,(B(\lambda)-p_h B(\lambda))v_h)_V\\
%			& \leq & 	\Vert (B(\lambda)-B_h(\lambda))v_h \Vert_V \Vert (I-p_h)(B(\lambda)v_h) \Vert_V. 
%		\end{array}
%	\end{equation}
Using the property of the projection operator $p_h$, it holds that
	\begin{equation}
		\Vert (B(\lambda)-B_h(\lambda))v_h \Vert_V \leq \Vert (I-p_n)(B(\lambda)v_h) \Vert_V \to 0.
	\end{equation}
\end{proof}

Let $\Omega_0 \subseteq \Omega$ be compact such that $\Gamma:=\partial \Omega_0 \subseteq \rho(F)$ and 
	$  \Omega_0 \cap \sigma (F) = \{ \lambda_0 \} $. 
The consistency error of the finite element method $\epsilon_h$ is simply
	\begin{equation}{\label{epsil}}
		\epsilon_h = \sup_{\eta \in \Gamma} \max_{g \in G(F,\lambda_0),\, \Vert g \Vert_V=1} \Vert F_h(\eta)p_h g - p_h F(\eta)g \Vert_V ,
	\end{equation}
where $G(F,\lambda_0)$ denotes the generalized eigenspace for $F(\cdot)$ at $\lambda_0$.

\begin{Theorem}
		 Let $\lambda_0$ be an eigenvalue of $F$ and $h$ be small enough.
%There exists an eigenvalue $\lambda_h$ of $F_h$ such that 
If $\lambda_0$ is semi-simple,
\begin{equation}
			\vert \lambda_h - \lambda_0 \vert  \leq C \epsilon_h \leq C h^{2-\epsilon-s}, \, \forall \lambda_h \in \sigma(F_h) \cap \Omega_0,
		\end{equation} 
and if $\lambda_0 $ is not semi-simple,
	\begin{equation}
		\vert \lambda_h - \lambda_0 \vert  \leq C h^{(1-s)/\kappa}, \, \forall \lambda_h \in \sigma(F_h) \cap \Omega_0,
	\end{equation} 
where $\epsilon \to 0_+$, $s \to 0_+$  and $ \kappa = \kappa(F,\lambda_0)$. %is order of the pole $\lambda_0$ of the operator function $F^{-1}(\cdot)$. 
\end{Theorem}

\begin{proof}
Due to Lemmas 2, 3, and 4, we only need to verify (A4) in Theorem 1, 
namely, the sequence $\{F_h(\lambda)\}$ is regular for every $\lambda \in \Omega_0$. 
Let $\{h_n \}$ be a positive number sequence which goes to $0$ monotonically when $n \in \mathbb{N} \rightarrow \infty$. 
Consider an arbitrary subsequence $\mathbb{N}' \subseteq \mathbb{N}$  and $\{ v_{h_{n'}}\} \subset \{ v_{h_n}\}, \, n' \in \mathbb{N}' $. 
According to the definition of $\mathcal{P}$-compactness, for any $\lambda$, 
there exists a subsequence $\{ F_{h_{n''}} (\lambda) v_{h_{n''}}\}, \, n'' \in \mathbb{N}'' \subseteq \mathbb{N}' $ and some $y \in V $ such that $ \Vert F_{h_{n''}}(\lambda)v_{h_{n''}} - p_{h_{n''}} y \Vert_V \to 0 $. 
The goal is to show the existence of some $v \in V$ such that  $\Vert v_{h_{n''}} - p_{h_{n''}} v \Vert_V \rightarrow 0$ as $n'' \in \mathbb{N}'' \to \infty$. 		\begin{enumerate}
		\item If $\lambda \in \rho (F),$ then $ F(\lambda)^{-1}$ exists and is bounded. Letting $v= F(\lambda)^{-1} y,$ due to Lemma~\ref{vh} and Lemma~\ref{order}, we have that
%			\begin{equation}
%			\begin{array}{lll}
%				\Vert v_{h_{n''}} -p_{h_{n''}} v \Vert_V & = & \Vert F_{h_{n''}}(\lambda)^{-1}[F_{h_{n''}}(\lambda) v_{h_{n''}}-p_{h_{n''}} F(\lambda) v+p_{h_{n''}} F(\lambda) v-F_{h_{n''}}(\lambda) p_{h_{n''}} v ]\Vert_V \\
%				& \leq & \Vert F_{h_{n''}}(\lambda)^{-1} \Vert (\Vert F_{h_{n''}}(\lambda) v_{h_{n''}}-p_{h_{n''}} F(\lambda) v \Vert_V +  \Vert p_{h_{n''}} F(\lambda) v-F_{h_{n''}}(\lambda) p_{h_{n''}} v \Vert_V ) \\
%				& \leq & \Vert F_{h_{n''}}(\lambda)^{-1} \Vert (\Vert F_{h_{n''}}(\lambda) v_{h_{n''}}-p_{h_{n''}} y \Vert_V +  \Vert p_{h_{n''}} F(\lambda) v-F_{h_{n''}}(\lambda) p_{h_{n''}} v \Vert_V ) \\
%				& \to & 0.
%			\end{array}
%		\end{equation}
		\begin{eqnarray*}
			&&\Vert v_{h_{n''}} -p_{h_{n''}} v \Vert_V \\& = & \Vert F(\lambda)^{-1}[(F(\lambda)-F_{h_{n''}}(\lambda))(v_{h_{n''}}-p_{h_{n''}} v)+F_{h_{n''}}(\lambda)(v_{h_{n''}} - p_{h_{n''}} v)] \Vert_V \\
			& = & \Vert F(\lambda)^{-1}[(F(\lambda)-F_{h_{n''}}(\lambda))(v_{h_{n''}}-p_{h_{n''}} v) + F_{h_{n''}}(\lambda) v_{h_{n''}} - p_{h_{n''}} y   \\ 
			& & \qquad +p_{h_{n''}} F(\lambda) v -F_{h_{n''}}(\lambda) p_{h_{n''}} v ]\Vert_V \\
			& \leq & \Vert F(\lambda)^{-1} \Vert_V (\Vert (F(\lambda)-F_{h_{n''}}(\lambda))(v_{h_{n''}}-p_{h_{n''}} v) \Vert_V   \\
			& & \qquad +  \Vert F_{h_{n''}}(\lambda) v_{h_{n''}}  - p_{h_{n''}} y \Vert_V  +\Vert p_{h_{n''}} F(\lambda) v - F_{h_{n''}}(\lambda) p_{h_{n''}} v \Vert_V ) \\
			& \to & 0.
		\end{eqnarray*}

	\item  If $\lambda \in \sigma(F)$, ${\mathcal N}(F(\lambda)) $ is finite-dimensional since $F(\lambda)$ is Fredholm. In fact,
	\begin{equation}
		\Vert F(\lambda) v_{h_{n''}} - y \Vert_V = \Vert F(\lambda) v_{h_{n''}} - F_{h_{n''}}(\lambda) v_{h_{n''}} +F_{h_{n''}}(\lambda) v_{h_{n''}} - p_{h_{n''}} y + p_{h_{n''}} y -y \Vert_V \to 0, 
	\end{equation}
	and thus $ y \in {\mathcal R}(F(\lambda))$ because ${\mathcal R}(F(\lambda))$ is closed. 
	$F(\lambda)$ is invertible as a mapping from ${V/ {\mathcal N}(F(\lambda))}$ to ${\mathcal R}(F(\lambda))$ 
	and $F(\lambda)^{-1}$ is well-defined. Let $v \in {V/ {\mathcal N}(F(\lambda))} $ such that $F(\lambda)v =  y$.
	Then $\Vert v_{h_{n''}} - p_{h_{n''}} v \Vert_V \rightarrow 0$ can be proved as before.  
\end{enumerate}

%Consequently, (A1)-(A4) in Theorem \ref{Karma} are verified. 
Next we consider the consistency error $\epsilon_h$. The generalized eigenspace $G(F,\lambda_0)$ is finite dimensional since 
$F(\lambda_0)$ is a Fredholm operator. 
If $\lambda_0$ is semi-simple, for $g \in G(F,\lambda_0)$, we have $g \in H_0^2(D)\cap H^{3-\epsilon}(D)$ and $\kappa=1$. 
Combining (\ref{epsil}) and Lemma \ref{order},  we obtain that
 	\begin{equation}{\label{eps}}
 	\epsilon_h = \sup_{\eta \in \Gamma} \max_{g \in G(F,\lambda_0),\, \Vert g \Vert_V=1} \Vert F_h(\eta)p_h g - p_h F(\eta)g \Vert_V  \leq C h^{2-\epsilon-s},
 \end{equation}
where $\epsilon$ and $s$ are small enough. Especially, letting $ \epsilon \to 0 $, we obtain 
 	\begin{equation}
	\vert \lambda_h - \lambda_0 \vert  \leq C h^{2-s},\, \forall s \to 0_+ .
\end{equation} 
If $\lambda_0$ is not semi-simple, the convergence is given by
 	\begin{equation}
 	\vert \lambda_h - \lambda_0 \vert  \leq C h^{(1-s)/\kappa},\, \forall s \to 0_+.
 \end{equation} 
The proof is complete.

\end{proof}

The rest of this section is devoted to the spectral indicator method (SIM) to compute the eigenvalues of $F_h(\lambda)$ \cite{25, huang2018}.
Without loss of generality, let $\Omega \in \mathbb{C} $ be a square and $\Gamma:=\partial \Omega \subset \rho(F_h)$. 
Denote by $\aleph_h: V_h \to V_h$ the spectral projection 
\begin{equation}{\label{spectral operator}}
	\aleph_h = \frac{1}{2 \pi i} \int_{\Gamma} F_h(\eta)^{-1} d\eta.
\end{equation}
The matrix form of $F_h(\eta)$ is given by
\begin{equation}
	F_h(\eta) = A_h -  \eta B_h +\frac{\eta \sigma}{\eta - \sigma} C_h,
\end{equation}
where $A_h, B_h$, and $C_h$ are, respectively, the matrices corresponding to
\[
\int_{D} R(x) [ (\partial_{xx} u_h +\partial_{yy} u_h  )(\partial_{xx} v_h +\partial_{yy} v_h) + (1-\nu )(2 \partial_{xy} u_h \partial_{xy} v_h - \partial_{xx} u_h \partial_{yy} v_h -\partial_{yy} u_h \partial_{xx} v_h)] dx,
\]
$\int_{D} \rho d u_h v_h dx$, and $M u_h(x_0)v_h(x_0)$.
% Note that the requirement that  $x_0$ a mesh node makes it simper to construct $C_h$. 

If $\Gamma$ encloses no eigenvalues of $F_h$, the integral $ \int_{\Gamma} F_h(\eta)^{-1} ds$ should be zero as well as $\aleph_h(\vec{y}_h)$ for any vector $\vec{y}_h $ in $V_h$. In practice, one selects a random vector $\vec{y}_h \in V_h$ and solves $\vec{x}_h(\eta_i) \in V_h$ for $F_h(\eta_i) \vec{x}_h(\eta_i) = \vec{y}_h$. Gaussian quadrature can be used to approximate (\ref{spectral operator})
\begin{equation}
	\mathcal{I}_\Omega := \left\vert \frac{1}{2 \pi i} \sum_{i=1}^{m} w_i \vec{x}_h(\eta_i) \right\vert, 
\end{equation}
where $m$ is the number of quadrature nodes and $w_i, \, \eta_i $ represent the weights and corresponding nodes. 
One decides whether $\Omega$ contains eigenvalues using $\mathcal{I}_\Omega$. 
This is done in SIM by setting a threshold $\alpha$. If $\mathcal{I}_\Omega $ is less than $\alpha$, $\Omega$ contains no eigenvalues. 
Otherwise, one uniformly splits $\Omega$ into four squares $\Omega_1, \Omega_2, \Omega_3, \Omega_4$ and compute 
$ \mathcal{I}_{\Omega_1}, \, \mathcal{I}_{\Omega_2}, \, \mathcal{I}_{\Omega_3}, \,\mathcal{I}_{\Omega_4}$ correspondingly. 
The procedure continues until the size of the square is less than a given precision $\beta$. The center of the square is the
computed eigenvalues of $F_h(\cdot)$. We refer the reader to \cite{30, 24, Xiao2021JSC} for the details of the algorithm.

\section{Numerical Examples}
We present some numerical results using
a series of uniformly refined meshes $\{\mathcal{T}_{h_i}\}$, where $h_i (=h_{i-1}/2)$ is the mesh size.
To measure the convergence order, we use the relative error
\begin{equation}
  Rel.Err= \frac{\vert \lambda_{h_i} -\lambda_{h_{i-1}} \vert }{\lambda_{h_i}},
\end{equation}  
where $\lambda_{h_i}$ is an eigenvalue computed on mesh $\mathcal{T}_{h_i}$. 
For all examples, $R, \nu, \rho, d$ are constants and (\ref{Lux})-\eqref{Bux} can be written as
\begin{subequations}
	\begin{align}
	\Delta^2 u(x)&=\lambda \frac{\rho d}{R} u +  \sum_{j=1}^p \frac{\lambda \sigma_j}{R(\sigma_j - \lambda)} M_j \delta(x-x_j)u, \, x \in D,\\[1mm]
	u(x)&=\frac{\partial u}{\partial n} =0,\, x \in \partial D.
	\end{align}
\end{subequations}
%
%\begin{equation}
%	\begin{array}{l}
%		\Delta^2 u(x)=\lambda \frac{\rho d}{R} u +  \sum_{j=1}^p \frac{\lambda \sigma_j}{R(\sigma_j - \lambda)} M_j \delta(x-x_j)u, \, x \in D ,\\
%		u(x)=\frac{\partial u}{\partial n} =0,x \in \partial D.
%	\end{array}
%\end{equation}

\begin{Example}
	$\boldsymbol{p=1 \textrm{ \bf for a rectangular domain}}$ 
\end{Example}
Let $D=[0,1] \times [0,1], \,R=1, \, \rho d = 1, \, M =0.01, \, K=100, \, \sigma = 10000, \, x_0 = (9/26,19/26)^T $ as in \cite{17},
i.e., 
	\begin{subequations}{\label{ex1}}
	\begin{align}
		\Delta^2 u(x)&=\lambda u +\frac{100\lambda}{10000-\lambda} \delta(x-x_0) u  , \,  x \in D,\\[1mm]
		u(x)&=\frac{\partial u}{\partial n} =0, \, x \in \partial D.
	\end{align}
\end{subequations}
%
%\begin{equation}
%	\begin{array}{l}{\label{ex1}}
%		\Delta^2 u(x)=\lambda u +\frac{100\lambda}{10000-\lambda} \delta(x-x_0) u  , \,  x \in D ,\\
%		u(x)=\frac{\partial u}{\partial n} =0,x \in \partial D. \\
%	\end{array}
%\end{equation}

Fig.~\ref{Err1} shows the relative errors v.s. the degrees of freedom for the first five eigenvalues. 
% In contrast, our method computes the eigenvalues and the convergence orders in a direct manner. 
In Table.~\ref{rect}, we show the first three eigenvalues and convergence orders.
The convergence orders validate the theory, namely, at least 2 but no more than 4. 
The eigenvalues are consistent with those in \cite{17}.
\begin{figure}[h]
		\small
	\centering
	\includegraphics[scale=0.6]{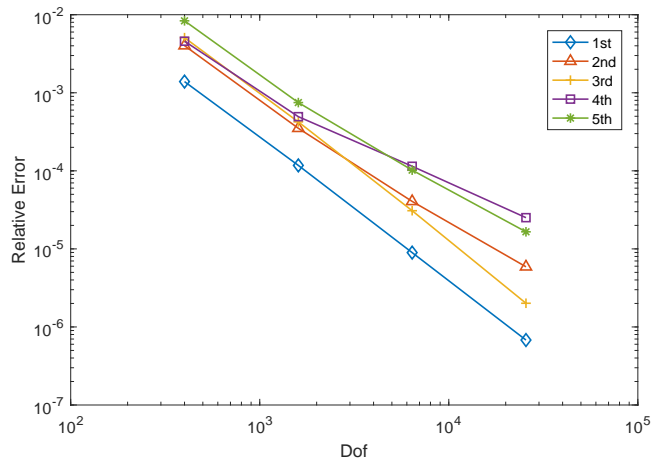}
	\caption{ Relative errors of first five eigenvalues for (\ref{ex1}). }
	{\label{Err1}}
\end{figure}
	\begin{table}[h]
		\centering
		\begin{tabular}{|c|c|c|c|c|c|c|}
			\hline
			$h$ &$\lambda_h(1st)$    & $order$ &$\lambda_h(2nd)$    & $order$ &$\lambda_h(3rd)$    & $order$  \\
			\hline
			1/5    &1273.40439&- &4847.23905   &- &5416.51831  &-\\
			1/10   &1271.63622 &-   &4827.85974    &-&5389.11595   &-\\
			1/20   &1271.48701   &3.57   &4826.16599 &3.52 &5386.83431   &3.59\\
			1/40   &1271.47561   &3.71 &4825.97011    &3.11  &5386.66812   &3.78 \\
			1/80   &1271.47475  &3.72  &4825.94165   &2.78 &5386.65729   &3.94\\
			\hline
		\end{tabular}
		\caption{Example 1: The first three eigenvalues and convergence orders.}
			{\label{rect}}
	\end{table}

\begin{Example}
	$\boldsymbol{p=1 \textrm{ \bf for an L-shaped domain}}$ 
\end{Example}
Consider an L-shaped domain $D = [-1 ,1] ^2 \setminus (0,1] \times [-1,0) $, which has a reentrant corner. Let $  M =0.01, \, K=20, \, \sigma = 2000, \, x_0 = (1/2,1/2)^T $. The other parameters are the same as the previous example. Fig.~\ref{ErrL} shows the relative errors v.s. the degrees of freedom for the first five eigenvalues. Table.~\ref{Lshape} shows  the first three eigenvalues and convergence orders. As expected, the convergence orders are lower than the previous example due to the reentrant corner. 
\begin{figure}[h]
	\small
	\centering
	\includegraphics[scale=0.6]{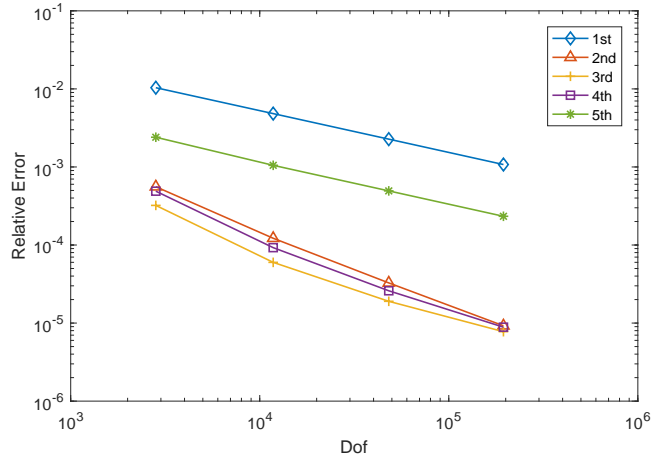}
	\caption{ Relative errors of first five eigenvalues for L-shaped domain.}
	{\label{ErrL}}
\end{figure}
\begin{table}[h]
	\centering
	\begin{tabular}{|c|c|c|c|c|c|c|}
		\hline
		$h$ &$\lambda_h(1st)$    & $order$ &$\lambda_h(2nd)$    & $order$ &$\lambda_h(3rd)$    & $order$  \\
		\hline
		1/8    &426.53874  &-  &678.31742    &- &902.57169   &-\\
		1/16   &422.16617   &-   &677.93923    &-&902.28136   &-\\
		1/32   &420.13893   &1.10   &677.85652 &2.19 &902.22727   &2.42\\
		1/64   &419.18665   &1.09 &677.83435    &1.90  &902.21013   &1.66 \\
		1/128   &418.73687  &1.08  &677.82809   &1.82 &902.20313   &1.29\\
		\hline
	\end{tabular}
	\caption{Example 2: The first three eigenvalues and convergence orders.}
	{\label{Lshape}}
\end{table}

\begin{Example}
 $\boldsymbol{p \neq 1}$ 
\end{Example}
Now we consider the case when $p = 2$, i.e., two added masses. Let $D = [0,1] \times [0,1] $ and
set $ \rho d = 1, \, M_1 =0.01, \, K_1=20, \, \sigma_1 = 2000, \, x_1 = (0.4,0.2)^T, \, M_2 =0.01, \, K_2=40, \, \sigma_2 = 4000, \, x_2 = (0.8,0.8)^T$. 
The equation is given by
	\begin{subequations}{\label{ex2}}
		\begin{align}
		\Delta^2 u(x)&=\lambda u +\frac{20\lambda}{2000-\lambda} \delta(x-x_1) u 
		 +\frac{40\lambda}{4000-\lambda} \delta(x-x_2)  u  , \,  x \in D, \\[1mm]
		u(x)&=\frac{\partial u}{\partial n} =0,\, x \in \partial D.
	\end{align}
%	\begin{equation}
%		\Delta^2 u(x)=\lambda u +\frac{20\lambda}{2000-\lambda} \delta(x-x_1) u 
%		 +\frac{40\lambda}{4000-\lambda} \delta(x-x_2)  u  , \,  x \in D,
%	\end{equation}
%	\begin{equation}
%		u(x)=\frac{\partial u}{\partial n} =0,\, x \in \partial D.
%	\end{equation}
\end{subequations}
%
%\begin{equation}
%	\begin{array}{l}{\label{ex2}}
%		\Delta^2 u(x)=\lambda u +\frac{20\lambda}{2000-\lambda} \delta(x-x_1)
%		 +\frac{40\lambda}{4000-\lambda} \delta(x-x_2)  u  , \,  x \in D ,\\
%		u(x)=\frac{\partial u}{\partial n} =0,x \in \partial D. \\
%	\end{array}
%\end{equation}

In Fig.~\ref{Err2}, the relative errors v.s. the degrees of freedoms are shown. Table.~\ref{rect2} lists the first three eigenvalues and convergence orders.
Again, the convergence orders are between $2$ and $4$.

\begin{figure}[h]
	\small
	\centering
	\includegraphics[scale=0.6]{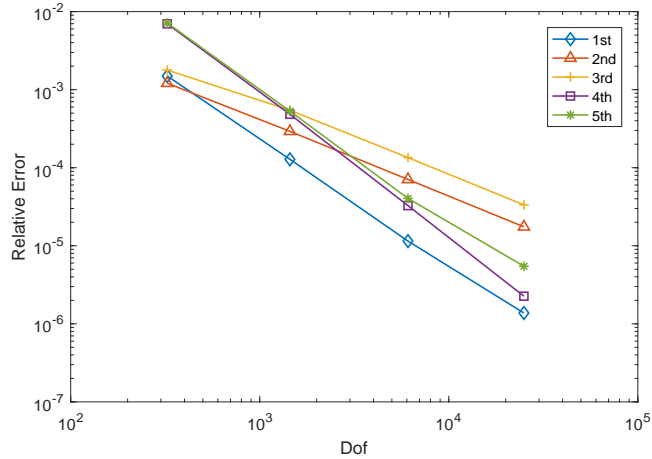}
	\caption{ Relative errors of first five eigenvalues for (\ref{ex2}). }
	{\label{Err2}}
\end{figure}

\begin{table}[h]
	\centering
	\begin{tabular}{|c|c|c|c|c|c|c|}
		\hline
		$h$ &$\lambda_h(1st)$    & $order$ &$\lambda_h(2nd)$    & $order$ &$\lambda_h(3rd)$    & $order$  \\
		\hline
		1/5    &1969.78685  &-  &3713.12915    &- &5437.59842   &-\\
		1/10   &1967.40653   &-   &3706.51164   &-&5399.91834   &-\\
		1/20   &1966.83124   &2.05   &3704.49126 &1.71 &5397.30692   &3.85\\
		1/40   &1966.69214   &2.05 &3703.99197    &2.02  &5397.13054   &3.89 \\
		1/80   &1966.65774  &2.02  &3703.86826   &2.01 &5397.11834   &3.85\\
		\hline
	\end{tabular}
	\caption{Example 3: The first three eigenvalues and corresponding convergence orders.}
	{\label{rect2}}
\end{table}

\section{Conclusion}
In the paper, we develop a new numerical method for the nonlinear eigenvalue problems associated to the vibrations of the plate-spring-load system. 
The problem is formulated as the eigenvalue problem of a holomorphic Fredholm operator function, which is discretized using the BFS element.
The spectral indicator method is then employed to compute the eigenvalues.
The convergence of the discrete eigenvalues is proved and validated by several numerical examples. 

% However, the verified convergence seems to be suboptimal for the precision of convergence order. Indeed, it is challenging to identify which Hilbert space each generalized eigenfunction belongs to. On the other hand, improving speed of this proposed algorithm is another orientation we need to focus on. We now are trying to run it in parallel simultaneously. 

The proposed method is effective for problems with nonlinear dependence on the eigen-parameter, which
has been successfully used to compute the Dirichlet eigenvalues, the transmission eigenvalues, 
and the band structures of photonic crystals \cite{30, 24, Xiao2021JSC}. 
In the future, we plan to extend the method to treat damped vibrations. Extension of the method to treat 3D problems on general domains
is another interesting topic.


\begin{thebibliography}{99} 
	\bibitem{17}{S.I. Solov$'${\"e}v. Eigenvibrations of a plate with elastically attached load. \emph{Citeseer, 2003.}}
	
	\bibitem{19}{M. Febbo, D.V. Bambill, and R.E. Rossi. Free vibration of a rectangular plate with an attached three-degree-of-freedom spring-mass system. \emph{Structural Engineering and Mechanics}, 40(5):637-654, 2011.} 
	
	\bibitem{21}{S. Chakraverty. \emph{Vibration of Plates}. CRC press, Boca Raton, 2008.}
	
	% \bibitem{8}{J. Awrejcewicz, V.A. Krys'ko, and A.F. Vakakis. Vibration of plates and shells with added masses. in \emph{Nonlinear Dynamics of Continuous Elastic Systems}(1-91), Springer-Verlag, Berlin, Heidelberg, 2004.}
	\bibitem{8}{J. Awrejcewicz, V.A. Krys'ko, and A.F. Vakakis. \emph{Nonlinear Dynamics of Continuous Elastic Systems}. Springer-Verlag, Berlin, Heidelberg, 2004.}
	
	\bibitem{33}{T. Betcke, N.J. Higham, V. Mehrmann, C. Schr$\ddot{o}$der, and F. Tisseur. NLEVP: A collection of nonlinear eigenvalue problems. \emph{ACM Transactions on Mathematical Software}, 39(2): 1-28, 2013.}
	
	\bibitem{5}{V. Mehrmann and H. Voss. Nonlinear eigenvalue problems: A challenge for modern eigenvalue methods. \emph{GAMM-Mitteilungen}, 27(2):121-152, 2004.}  
	
	\bibitem{1}	
	{	H. Voss. An Arnoldi method for nonlinear eigenvalue problems. \emph{BIT Numerical Mathematics}, 44(2):387-401, 2004.
	
	\bibitem{3}{D. Kressner. A block Newton method for nonlinear eigenvalue problems. \emph{Numerische Mathematik}, 114(2):355-372, 2009.}

    \bibitem{7}{Y. Su and Z. Bai. Solving rational eigenvalue problems via linearization. \emph{SIAM Journal on Matrix Analysis and Applications}, 32(1):201-216, 2011. }

	\bibitem{28}{A. Ruhe. Algorithms for the nonlinear eigenvalue problem. \emph{SIAM Journal on Numerical Analysis,} 10(4):674-689, 1973. }

	\bibitem{26}{O. Karma. Approximation in eigenvalue problems for holomorphic Fredholm operator functions I. \emph{Numerical Functional Analysis and Optimization}, 17(3-4):365-387, 1996. }
		
	\bibitem{27}{O. Karma. Approximation in eigenvalue problems for holomorphic Fredholm operator functions Ii (Convergence Rate). \emph{Numerical Functional Analysis and Optimization}, 17(3-4):389-408, 1996.} 
		
	\bibitem{12}{ J. Valdman. MATLAB Implementation of $C_1$ finite elements: Bogner-Fox-Schmit rectangle. \emph{International Conference on Parallel Processing and Applied Mathematics}(pp. 256-266), 2019.}

    \bibitem{29}{J. Sun and A. Zhou. \emph{Finite Element Methods for Eigenvalue Problems}. CRC Press, Boca Raton, 2016.}
	\bibitem{14}{ A.B. Andreev, R.D. Lazarov, and M.R. Racheva. Postprocessing and higher order convergence of the mixed finite element approximations of biharmonic eigenvalue problems. \emph{Journal of Computational and Applied Mathematics}, 182(2):333-349, 2005.}
			
	\bibitem{22}{P. Grisvard. \emph{Singularities in Boundary Value Problems}. Springer Verlag, Berlin, 1992.}
	
	\bibitem{BS2005} S. Brenner and L. Sung. $C^0$ interior penalty methods for fourth order elliptic boundary value problems on polygonal domains. \emph{ Journal of Scientific Computing} 22/23 (2005), 83-118.
   
    \bibitem{15}{K.M. Kalayeh, J.S. Graf, and M.K. Gobbert. FEM convergence for PDEs with point sources in 2-D and 3-D. In \emph{Proceedings of the COMSOL Conference}, 2015. }

   
		
		
	\bibitem{31}{P.G. Ciarlet.	\emph{The finite element method for elliptic problems}. SIAM, Philadelphia, 2011.
	}
	
   \bibitem{32}{W. Hackbusch. \emph{Elliptic Differential Equations: Theory and Numerical Treatment}, 2nd Edition, Springer-Verlag, Berlin, 2017.}


			
   \bibitem{25} {R. Huang, A. Struthers, J. Sun, and R. Zhang. 
		 Recursive integral method for transmission eigenvalues.
			\emph{Journal of Computational Physics}, 327:830-840, 2016.}
			
	\bibitem{huang2018}{R. Huang, J. Sun, and C. Yang, 
				{ Recursive integral method with Cayley transformation.}
				\emph{Numerical Linear Algebra with Applications}, 25(6):e2199, 2018.}
				
	\bibitem{30}{W. Xiao, B. Gong, J. Sun, and Z. Zhang. A new finite element approach for the Dirichlet eigenvalue problem. \emph{Applied Mathematics Letters}, 105:106295, 2020.} 
		
	\bibitem{24}{B. Gong, J. Sun, T. Turner, and C. Zheng. Finite element approximation of transmission eigenvalues for anisotropic media. \emph{arXiv:2001.05340}, 2020.}
		
	\bibitem{Xiao2021JSC}{ W. Xiao, B. Gong, J. Sun, and Z. Zhang. 
			Finite element calculation of photonic band structures for frequency dependent materials. \emph{Journal of Scientific Computing}, 87(1):1-16, 2021.}}


%	\bibitem{2} 	{ Xin Huang, Zhaojun Bai, and Yangfeng Su. Nonlinear rank-one modification of the symmetric eigenvalue problem. \emph{  Journal of Computational Mathematics}, pages 218-234, 2010.}
	
%	\bibitem{4}{Yousef Saad, Mohamed El-Guide, and Agnieszka Miedlar. A rational approximation method for the nonlinear eigenvalue problem. \emph{arXiv preprint arXiv:1901.01188}, 2019.}
	
	
	
%	\bibitem{6}{Heinrich Vo{\ss}. Initializing iterative projection methods for rational symmetric eigenproblems. \emph{Preprints des Institutes f{\"u}r Mathematik}, 2003.}
	

	
%	\bibitem{8}{J. Awrejcewicz, V.A. Krysâko, and A.F. Vakakis. Vibration of Plates and Shells with Added Masses. \emph{Nonlinear Dynamics of Continuous Elastic Systems}, pages 1-91, 2004.}
	
	%\bibitem{9}{Y.-C. Das and D.-R. Navaratna. Vibrations of a rectangular plate with concentrated mass, spring, 1963.}
	
%	\bibitem{9}{ P.W. Lawrence, M. Van Barel, and P. Van Doore. Backward error analysis of polynomial eigenvalue problems solved by linearization. \emph{SIAM journal on Matrix Analysis and Applications}, 37(1):123-144, 2016.}
	
%	\bibitem{10}{T.P. Chang and C.Y. Chang. Vibration analysis of beams with a two degree-of-freedom spring-mass system. \emph{International journal of solids and structures}, 1998.}
%	
%	
%	\bibitem{11}{H. Blum, R. Rannacher, and R. Leis. On the boundary value problem of the biharmonic operator on domains with angular corners. \emph{Mathematical Methods in the Applied Sciences}, 2(4):556-581, 1980.} 
	

%	
%	\bibitem{13}{J-S Wu, H-M Chou, and D-W Chen. Free vibration analysis of a rectangular plate carrying multiple various concentrated elements. \emph{Proceedings of the Institution of Mechanical Engineers, Part K: Journal of Multi-body Dynamics}, 217(2):171-183, 2003.}
	
	

	

	
%	\bibitem{16}{Lada Mazurenko and Heinrich Vo{\ss}. \emph{On the number of eigenvalues of a rational eigenproblem}. Technical report, 2003.}
	


	
%	\bibitem{18}{Heinrich Vo{\ss}. Eigenvibrations of a plate with elastically attached loads. In \emph{Preprints des Institutes f{\"u}r Mathematik}, 2004.}
	

	
%	\bibitem{20}{David Bindel and Amanda Hood. Localization theorems for nonlinear eigenvalue problems. \emph{SIAM Journal on Matrix Analysis and Applications}, 34(4): 1728-1749, 2013.} 
	

	

	
%	\bibitem{23}{Pierre Grisvard. \emph{Elliptic problems in nonsmooth domains}, SIAM, 2011.}
	
	
%	\bibitem{25}{Ruihao Huang, Allan A. Struthers, Jiguang Sun, and Ruming Zhang. Recursive integral method for transmission eigenvalues. \emph{Journal of Computational Physics}, 327: 830-840, 2016.} 


	
	
	

	

	
%	\bibitem{32}{ 
%		L.V. Andreev, A.L. Dyshko, and I.D. Pavlenko. Dynamics of plates and shells with concentrated masses. 
%		\emph{Mashinostroenie, Moscow,} 
%		pages 56--60, 1988.}
% \bibitem{Sun2011SIAMNA} J. Sun, 	{\em Iterative methods for transmission eigenvalues}. SIAM J. Numer. Anal. \textbf{49} (2011), no. 5, 1860--1874.


%\bibitem{Huang2020} R. Huang, J. Sun and C. Yang,
%	{\em A multilevel spectral indicator method for eigenvalues of large non-Hermitian matrices.}
%	 CSIAM Trans. Appl. Math. 1 (2020), no. 3, 463-477.	

\end{thebibliography}
\end{document}